\documentclass[
final
 , nomarks
]{dmtcs-episciences}

\usepackage[utf8]{inputenc}
\usepackage{subfigure}
\usepackage[round]{natbib}

\usepackage{amsfonts, amsmath, amssymb, amsthm, url, enumitem,tikz}
\usetikzlibrary{patterns}

\theoremstyle{definition}
\newtheorem{theorem}{Theorem}[section]
\newtheorem{lemma}[theorem]{Lemma}
\newtheorem{corollary}[theorem]{Corollary}
\newtheorem{proposition}[theorem]{Proposition}

\def\D{\mathcal{D}}
\def\Z{\mathcal{Z}}
\def\Q{\mathcal{Q}}
\def\XX{\mathcal{X}}
\def\YY{\mathcal{Y}}
\def\X{\mathbf{X}}
\def\B{\mathbf{B}}
\def\U{\mathbf{U}}
\def\J{\mathbf{J}}
\def\ve{\varepsilon}

\newcommand{\abs}[1]{\left\vert#1\right\vert}

\author{Margaret-Ellen Messinger\thanks{Supported by the Natural Sciences and Engineering Research Council of Canada (Discovery Grant 2018-04059).}
  \and Logan Pipes}
\title{On Pancyclicity in a Mixed Model for Domination Reconfiguration}
\affiliation{
  Mount Allison University, Sackville, Canada}
\begin{document}
\publicationdata{vol. 28:2}{2026}{21}{10.46298/dmtcs.15637}{2025-05-07; 2025-05-07; 2026-03-02}{2026-03-04}
\maketitle
\begin{abstract}
  A new model for domination reconfiguration is introduced which combines the properties of the preexisting token addition/removal (TAR) and token sliding (TS) models.  The vertices of the TARS-graph correspond to the dominating sets of $G$, where two vertices are adjacent if and only if they are adjacent via either the TAR reconfiguration rule or the TS reconfiguration rule.  While the domination reconfiguration graph obtained by using only the TAR rule (sometimes called the dominating graph) will never have a Hamilton cycle, we show that for some classes of graphs $G$, by adding a relatively small number of token sliding edges, the resulting graph is not only hamiltonian, but is in fact pancyclic.  In particular, if the underlying graphs are trees, complete graphs, or complete multipartite graphs, we show that their TARS-graphs will be pancyclic. Notably, we prove that if the TARS-graphs of $G$ and $H$ are pancyclic, then the TARS-graph of the join $G \vee H$ will also be pancyclic. We conclude by posing the question: Are all TARS-graphs pancyclic?
\end{abstract}


\section{Introduction}

The study of reconfiguration concerns the solutions to a problem and the relationships between those solutions. A reconfiguration graph can be constructed by representing each solution by a vertex, with two vertices being adjacent if their corresponding solutions are deemed similar according to some rule. In most contexts, the fundamental questions in reconfiguration pertain to connectivity: Is it possible to transition from one solution to another given solution? Is it possible to start at any one solution, transition through solutions, and arrive at any other solution? Another central question about the structure of reconfiguration graphs concerns hamiltonicity; that is, under what conditions does a reconfiguration graph have a Hamilton path or Hamilton cycle? For a more detailed introduction to graph reconfiguration, see the surveys by~\cite{MN20} and~\cite{Nishimura18}.

A \emph{dominating set} of a graph $G$ is a set $D \subseteq V(G)$ such that every vertex of $V(G)\backslash D$ is adjacent to at least one vertex of $D$.
In the \emph{dominating graph} of $G$, $\D(G)$, each vertex represents a dominating set of $G$. Adjacency in $\D(G)$ is defined as follows: Distinct vertices $x$ and $y$ of $\D(G)$ (with corresponding dominating sets $X$ and $Y$, respectively) are adjacent if and only if $Y$ can be obtained from $X$ by adding or removing a single vertex. The graph $\D(G)$ is the reconfiguration graph of dominating sets of $G$ under the \emph{token addition/removal} (TAR) model, first considered by~\cite{HS14}. \cite{ABCHMSS21} observed that since dominating graphs have an odd number of vertices but no odd cycles, no dominating graph has a Hamilton cycle. However, \cite{ABCHMSS21,ABCHMSS22} proved that some classes of seed graphs yield dominating graphs with Hamilton paths. In this paper, we show for some classes of graphs $G$ that by adding only a few additional edges to the dominating graph $\D(G)$, we find not only a Hamilton cycle, but cycles of every cardinality up to $\abs{V(\D(G))}$. That is, by adding a relatively small number of edges to $\D(G)$, the graph becomes \emph{pancyclic}.
We call the new graph the TARS-graph of $G$, since the adjacencies are given by both the token addition/removal (TAR) model and the token sliding (TS) model, which we define in Section~\ref{sec:defn}.

The primary objective of this work is to establish the pancyclicity of many of these TARS-graphs, and not to explicitly count the number of TS edges used.
We occasionally note the relative proportions of edge types in the constructed cycles,
though we make no claim that these are optimal.
For instance, in the TARS-graph depicted in Figure~\ref{fig:TARS_example},
the graph would be pancyclic even with only one TS-edge (the one depicted horizontally).
However, the proof of Theorem~\ref{thm:joins} uses more than one TS edge to construct all the cycles.

To our knowledge, this (and the undergraduate thesis by the second author,~\cite{Pipes24}, from which this work is inspired) is the first time the union of two disjoint adjacency rules has been considered for reconfiguration graphs. However, we also draw the reader's attention to the recent work by~\cite{Beaton} which generalizes the token sliding adjacency rule in another way.

In Section~\ref{sec:defn}, we formally define the TARS-graph of $G$ and make some preliminary observations about the properties of TARS-graphs. In Section~\ref{sec:trees}, we show that the TARS-graph of any tree is pancyclic. To do this, we prove two more general results which show that if a graph $G$ has a pancyclic TARS-graph, then adding leaves in a prescribed way to $G$ will yield a new graph with a pancyclic TARS-graph. In Section~\ref{sec:join}, we show that if the TARS-graphs of $G$ and $H$ are pancyclic, then the TARS-graph of their join $G \vee H$ must also be pancyclic. Notably, this implies graphs such as complete graphs, threshold graphs, complete multipartite graphs, and complete split graphs have pancyclic TARS-graphs. We conclude with some open questions in Section~\ref{sec:questions}.


\subsection{Definitions and preliminary results}\label{sec:defn}

The \emph{token sliding} (TS) model for domination reconfiguration graphs uses a different rule for adjacency than the TAR model: Vertices $x,y$ in the reconfiguration graph of $G$ are adjacent if and only if an element of dominating set $X$ can be exchanged for an adjacent element to obtain the dominating set $Y$. In this model, an element from $X$ is said to ``slide'' along an edge of $G$ to yield $Y$. The TS model was introduced by~\cite{FHHH11}, who defined $G(\gamma)$ to be the reconfiguration graph in which each vertex corresponds to a minimum dominating set of $G$ and where the adjacencies are as defined here.

If we generalize this idea to include a vertex in the token sliding domination reconfiguration graph for each of the dominating sets of $G$---rather than just those of minimum cardinality---then the domination reconfiguration graph will be disconnected. This is because vertices corresponding to dominating sets of different cardinalities will never be adjacent.
In the TAR model, however, the dominating graph $\D(G)$ is never disconnected, as any superset of a dominating set is still a dominating set, and thus every vertex in $\D(G)$ lies on a path to the vertex representing $V(G)$, obtained by sequentially including the vertices of $G$ absent from the initial dominating set.

While these two models are very contrasting in this respect, we now introduce a novel mixed model which unifies them, called the TARS-graph.
For a graph $G$, called the \emph{seed} graph, each vertex of the \emph{TARS-graph} of $G$, $\ve(G)$, represents a dominating set of $G$. Let $x,y$ be vertices in $\ve(G)$ which represent dominating sets $X$ and $Y$ of $G$. Then $x$ and $y$ are adjacent in $\ve(G)$, written $x \sim y$, if and only if \begin{enumerate}
    \item $Y$ can be obtained from $X$ by adding or deleting a single vertex of $G$; or
    \item there exist vertices $u \in X$, $v \in Y \setminus X$ that are adjacent in $G$ such that \begin{math}Y = X \cup \{v\} \backslash \{u\}\end{math}.
\end{enumerate}

If the former condition holds, we say that $x$ and $y$ are adjacent via token addition/removal,
and in the latter case, we say that $x$ and $y$ are adjacent via token sliding.

By allowing both rulesets for adjacency, the TARS-graph $\ve(G)$ of a graph $G$ represents the union of the edge sets of the dominating graph $\D(G)$ and $G(\gamma)$, where the latter is extended to dominating sets of all cardinalities. By combining these two separate models, the reconfiguration graph gains the guaranteed connectivity of the TAR model,
while also allowing for the ability to transition between dominating sets of the same cardinality according to the TS model.
For an illustration of this joint model, see Figure~\ref{fig:TARS_example}, which depicts the TARS-graph of $K_{1,3}$, where the solid edges correspond to edges given by the TAR model and the dotted edges correspond to edges given by the TS model.
In particular, this choice of seed graph highlights that we include TS edges between dominating sets of all cardinalities; not just those of minimum cardinality, since there is a unique dominating set of $K_{1,3}$ of minimum cardinality.
Note that we will often abuse terminology and refer to dominating sets of $G$ interchangeably with vertices of $\ve(G)$ when it is clear from context whether we are referring to $G$ or to $\ve(G)$.

\begin{figure}[thbp]
	\centering
	\begin{tikzpicture}[scale=0.8]
		\newcommand\starThree[5]{%
			\draw[shift={#1}, thick, fill=white] (0,0) circle (1);%
			\draw[shift={#1}, thick] (0,0) -- ( 90:0.8);%
			\draw[shift={#1}, thick] (0,0) -- (210:0.8);%
			\draw[shift={#1}, thick] (0,0) -- (330:0.8);%
			\draw[shift={#1}, color=black, fill=#2, line width=1pt] (  0,  0) circle (4pt); 
			\draw[shift={#1}, color=black, fill=#3, line width=1pt] ( 90:0.7) circle (4pt); 
			\draw[shift={#1}, color=black, fill=#4, line width=1pt] (330:0.7) circle (4pt); 
			\draw[shift={#1}, color=black, fill=#5, line width=1pt] (210:0.7) circle (4pt); 
		}

		\draw[very thick] (0,3) -- (3.5,6);
		\draw[very thick] (0,3) -- (3.5,3);
		\draw[very thick] (0,3) -- (3.5,0);

		\draw[very thick] (3.5,6) -- (7,6);
		\draw[very thick] (3.5,6) -- (7,0);

		\draw[very thick] (3.5,3) -- (7,6);
		\draw[very thick] (3.5,3) -- (7,3);

		\draw[very thick] (3.5,0) -- (7,3);
		\draw[very thick] (3.5,0) -- (7,0);

		\draw[very thick] (7,6) -- (10.5,3);
		\draw[very thick] (7,3) -- (10.5,3);
		\draw[very thick] (7,0) -- (10.5,3);

		\draw[very thick] (10.5,6) -- (10.5,3);

		\draw[very thick, dashed] (10.5,6) -- (7,6);
		\draw[very thick, dashed] (10.5,6) -- (7,3);
		\draw[very thick, dashed] (10.5,6) -- (7,0);

		\starThree{(0,   3)}{red  }{black}{black}{black}
		\starThree{(3.5, 6)}{red  }{red  }{black}{black}
		\starThree{(3.5, 3)}{red  }{black}{red  }{black}
		\starThree{(3.5, 0)}{red  }{black}{black}{red  }
		\starThree{(7,   6)}{red  }{red  }{red  }{black}
		\starThree{(7,   3)}{red  }{black}{red  }{red  }
		\starThree{(7,   0)}{red  }{red  }{black}{red  }
		\starThree{(10.5,6)}{black}{red  }{red  }{red  }
		\starThree{(10.5,3)}{red  }{red  }{red  }{red  }
	\end{tikzpicture}
	\caption{The TARS-graph of $K_{1,3}$ where the vertices in each dominating set are coloured \textcolor{red}{red}.}
	\label{fig:TARS_example}
\end{figure}
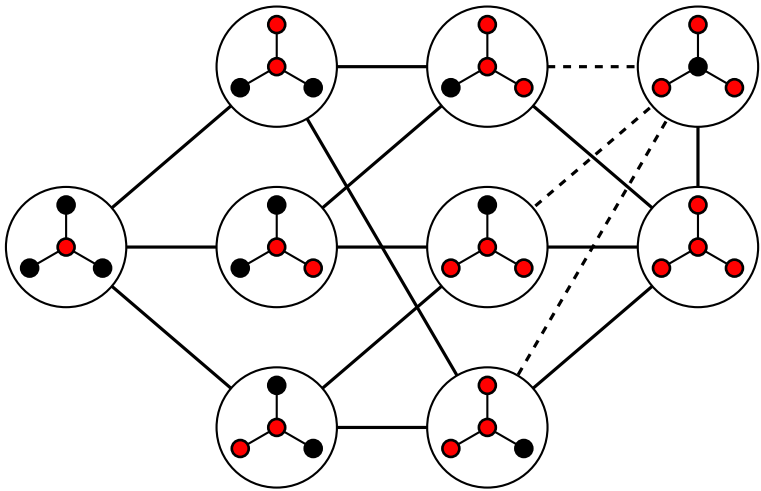

We provide several relevant definitions in the coming paragraphs, but we refer the reader to~\cite{BM08} for any undefined graph-theoretic terms.

A graph $G$ on $n$ vertices is \emph{pancyclic} if, for every integer $\ell$ satisfying $3 \leq \ell \leq n$, $G$ contains a cycle of length $\ell$. That is, if a graph has a cycle of every possible length up to and including a \emph{Hamilton cycle}, which is a cycle containing every vertex of the graph exactly once. A Hamilton cycle in a reconfiguration graph can be thought of as a \emph{combinatorial Gray code}: a listing of all the objects in a set so that successive objects differ in some prescribed minimal way.

The binary-reflected Gray code is a particularly useful type of combinatorial Gray code which will be a key tool in the proof of one of our main results (Theorem~\ref{thm:joins}).
More specifically, the \emph{binary-reflected Gray code on $n$ bits}
is a cycle through all of the bitstrings of length $n$
such that any two successive bitstrings differ by only a single bit.
The binary-reflected Gray code on $n$ bits can be generated recursively from the binary-reflected Gray code on $n-1$ bits as follows.
The first half of the code on $n$-bits is obtained by prefixing the entries of the $n-1$ sequence with a $0$.
The second half of the code on $n$-bits is obtained by reversing the sequence on $n-1$ bits and prefixing these entries with a $1$.
An illustration of the binary-reflected Gray code on $5$ bits is given in Figure~\ref{fig:gray_code}.
We refer the reader to \cite{Mutze23} for more information about the binary-reflected Gray code.

\begin{figure}[thbp]
	\centering
	\begin{tabular}{|rrrrrrr|}
		\hline
		00000 &~& 01100 &~& 11000 &~& 10100 \\
		00001 &~& 01101 &~& 11001 &~& 10101 \\
		00011 &~& 01111 &~& 11011 &~& 10111 \\
		00010 &~& 01110 &~& 11010 &~& 10110 \\
		00110 &~& 01010 &~& 11110 &~& 10010 \\
		00111 &~& 01011 &~& 11111 &~& 10011 \\
		00101 &~& 01001 &~& 11101 &~& 10001 \\
		00100 &~& 01000 &~& 11100 &~& 10000 \\ \hline
	\end{tabular}
	\caption{The binary-reflected Gray code on $5$ bits, read column-wise.}
	\label{fig:gray_code}
\end{figure}

This construction is extremely useful in the context of reconfiguration graphs.
In particular, it generalizes very naturally from a cycle through all of the bitstrings of some length to a cycle through all of the subsets of some set (a set of vertices, in our case),
where successive subsets differ by only a single missing/extra element.
That is, if we associate to each bit $j$ a vertex $v_j$ of $V(G)$, then the subset of $V(G)$ corresponding to a given bitstring contains $v_j$ if and only if bit $j$ is a $1$ in the bitstring. This binary-reflected Gray code then provides a way to iterate through many of the subsets of the vertices of a graph using only TAR edges.

An important question that arises with any new model/class of graphs is the question of characterization/realizability.
That is, is there some sort of characterization, structural or otherwise, that specifies which graphs can or cannot be TARS-graphs?
There is certainly no form of ``forbidden subgraph'' or ``forbidden induced subgraph'' characterization of TARS-graphs;
for any graph $G$, there is a TARS-graph which will not only contain $G$ as a subgraph, but in fact as an induced subgraph.
\cite{CHHH20} outlined a construction which shows that every graph $H$ is $G(\gamma)$ for some seed graph $G$, and that $G$ can be obtained from $H$ by adding only a few extra vertices and some edges.
Further, since minimum cardinality dominating sets will never be adjacent via token addition/removal (since they have the same cardinality), for the same seed graph $G$, the induced subgraph of $\ve(G)$ corresponding to only the minimum dominating sets will be exactly $H \cong G(\gamma)$.
Thus this same construction illustrates that any graph can be found as a subgraph of some TARS-graph.

However, there are still some restrictions on which graphs are realizable as TARS-graphs.
\cite{BCS09} proved that every graph $G$ has an odd number of dominating sets,
and therefore every TARS-graph will have an odd number of vertices.
Additionally, every TARS-graph will be connected, since any TARS-graph $\ve(G)$ will contain the dominating graph $\D(G)$ as a maximal subgraph, which is guaranteed to be connected.
So, while every graph can be a subgraph of a TARS-graph, not every graph can be a TARS-graph itself.


\section{Trees}\label{sec:trees}

Our first major result is to show that the TARS-graphs of trees are pancyclic.
If $G$ is a graph, and $v$ is a vertex in $G$, then we write $G - \{v\}$ to be the graph induced by the deletion of the vertex $v$.
Taking inspiration from \cite{ABCHMSS22}, we define two operations:
\begin{itemize}
	\item Operation A:
		Let $H$ be a graph with vertices $u$, $v$, and $x$ such that \begin{math}N(u) = \{x,v\}\end{math}
		and \begin{math}N(v) = \{u\}\end{math}.
		Then \begin{math}H^\prime = H - \{u,v\}\end{math} is obtained from $H$ by Operation A.
	\item Operation B:
		Let $H$ be a graph with vertices $u$, $v$, and $x$ such that \begin{math}N(u) = N(v) = \{x\}\end{math}.
		Then \begin{math}H^\prime = H - \{v\}\end{math} is obtained from $H$ by Operation B.
\end{itemize}
These operations are demonstrated pictorially in Figure~\ref{fig:operations}.

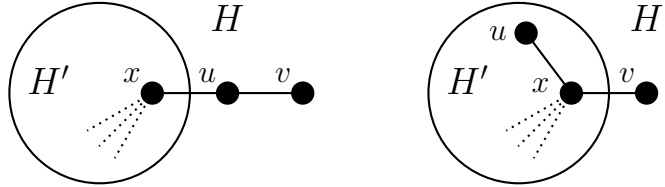
\begin{figure}[bht]
	\centering
		\begin{tikzpicture}
			\draw[thick, fill=white] (0,0) circle (1.2);
			\draw[fill=black, thick] (0.7,0) coordinate (x) circle (4pt) node[above left] {$x$} -- ++(1,0) circle (4pt) node[above left] {$u$} -- ++(1,0) circle (4pt) node[above left] {$v$};
			\node at (165:0.7) {\large $H^\prime$};
			\node at (1.7,1) {\large $H$};
			\foreach \th in {210, 225, 240} 
			\draw[dotted, thick] (x) -- ++(\th:1);
		\end{tikzpicture}
        \qquad
        \qquad
        \qquad
		\begin{tikzpicture}
			\draw[thick, fill=white] (0,0) circle (1.2);
			\draw[fill=black, thick] (0.7,0) coordinate (x) circle (4pt) node[left=4pt,yshift=3pt] {$x$} -- ++(0:1) circle (4pt) node[above left] {$v$};
			\draw[fill=black, thick] (x) -- ++(127:1) circle (4pt) node[left=3pt] {$u$};
			\node at (165:0.7) {\large $H^\prime$};
			\node at (1.7,1) {\large $H$};
			\foreach \th in {210, 225, 240}
			\draw[dotted, thick] (x) -- ++(\th:1);
		\end{tikzpicture}
	\caption{Operation A (left) and Operation B (right).}
	\label{fig:operations}
\end{figure}

In the following two theorems, we consider a graph $H^\prime$ to be the graph obtained by applying Operation A or Operation B to a graph $H$. We show that if $\ve(H^\prime)$ is pancyclic, then $\ve(H)$ is pancyclic.
For a set of vertices $S$ and another vertex $w$, it will be useful to abbreviate $S \cup \{w\}$ by the shorter $S^w$. We additionally generalize this notation for extensions by two or more elements with commas, such as taking $S^{u,v}$ to mean $S \cup \{u,v\}$, etc.


\begin{theorem}
	\label{thm:opA}
	Let $H^\prime$ be a connected graph on at least $2$ vertices such that it is obtained from another graph $H$ by Operation A.
	If $\ve(H^\prime)$ is pancyclic, then $\ve(H)$ is also pancyclic.
\end{theorem}

\begin{proof}
	Since $\ve(H^\prime)$ is pancyclic, there exist cycles in $\ve(H^\prime)$
	of every length from $3$ to $n = \abs{V(\ve(H^\prime))}$.
	Since each dominating set of $H^\prime$ can be extended into a dominating set of $H$
	by adding $v$ to the set (or $u$, or both),
	we also have cycles in $\ve(H)$ of each length from $3$ to $n$.
	Thus it suffices to find cycles in $\ve(H)$ of each length from $n+1$ to $\abs{V(\ve(H))}$.
	In particular, let $(F_j)_{1 \leq j \leq n-1}$ and $(G_j)_{1 \leq j \leq n}$ be cycles of length $n-1$ and $n$ in $\ve(H^\prime)$, respectively,
	with corresponding cycles $(F_j^v)_{1 \leq j \leq n-1}$ and $(G_j^v)_{1 \leq j \leq n}$ in $\ve(H)$. We use these to construct cycles of lengths $n+1$ through $n+4$ as follows, with deviations listed in \textcolor{red}{red}:
	\begin{align*}
		&\text{length}~n+1: &&(F_1^v, \textcolor{red}{F_1^{u,v}}, \textcolor{red}{F_2^{u,v}}, F_2^v, F_3^v, F_4^v, \ldots), \\
		&\text{length}~n+2: &&(G_1^v, \textcolor{red}{G_1^{u,v}}, \textcolor{red}{G_2^{u,v}}, G_2^v, G_3^v, G_4^v, \ldots), \\
		&\text{length}~n+3: &&(G_1^v, \textcolor{red}{G_1^{u,v}}, \textcolor{red}{G_1^u}, \textcolor{red}{G_2^u}, G_2^v, G_3^v, G_4^v, \ldots), \\
		&\text{length}~n+4: &&(G_1^v, \textcolor{red}{G_1^{u,v}}, \textcolor{red}{G_1^u}, \textcolor{red}{G_2^u}, \textcolor{red}{G_2^{u,v}}, G_2^v, G_3^v, G_4^v, \ldots).
	\end{align*}

	Note the use of the token sliding edge in the cycle of length $n+3$ here: $(G_2^u, G_2^v)$.
	We now define two especially important cycles---a cycle of even length $n+5$ and another of odd length $n+6$ (since $n$ is inherently odd, see~\cite{BCS09}). Note again the token sliding edge in $\Z_e$, $(G_n^v, G_n^u)$:
	\begin{align*}
		\Z_e &= (\textcolor{red}{G_1^u}, \textcolor{red}{G_2^u}, \textcolor{red}{G_2^{u,v}}, \textcolor{red}{G_1^{u,v}}, G_1^v, G_2^v, G_3^v, G_4^v, \ldots, G_{n-1}^v, G_n^v, \textcolor{red}{G_n^u}), \\
		\Z_o &= (\textcolor{red}{G_1^u}, \textcolor{red}{G_2^u}, \textcolor{red}{G_2^{u,v}}, \textcolor{red}{G_1^{u,v}}, G_1^v, G_2^v, G_3^v, G_4^v, \ldots, G_{n-1}^v, G_n^v, \textcolor{red}{G_n^{u,v}}, \textcolor{red}{G_n^u}).
	\end{align*}
	The remaining cycles will all be modifications of these two sequences.

	For any positive integer $j \leq n-1$,
	observe that \begin{math}(G_j^v, G_j^{u,v}, G_{j+1}^{u,v}, G_{j+1}^v)\end{math} forms a path in $\ve(H)$,
	and so does \begin{math}(G_j^{u,v}, G_j^u, G_{j+1}^u, G_{j+1}^{u,v})\end{math}.
	From $\Z_e$, note that replacing the pair \begin{math}(G_j^v, G_{j+1}^v)\end{math}
	for some odd $j$ satisfying $3 \leq j \leq n-2$
	with the sequence \begin{math}(G_j^v, G_j^{u,v}, G_{j+1}^{u,v}, G_{j+1}^v)\end{math}
	increases the length of the cycle by two,
	and further replacing \begin{math}(G_j^{u,v}, G_{j+1}^{u,v})\end{math}
	by \begin{math}(G_j^{u,v}, G_j^u, G_{j+1}^u, G_{j+1}^{u,v})\end{math}
	increases the length by another two.

	Performing these replacements one after another
	for each allowable value of $j$ generates cycles of all even lengths
	from $n+5$ up to $3n-1$.
	Similar replacements can be done for $\Z_o$,
	generating all odd cycle lengths from $n+6$ up to $3n$.
	Let $\Z_e^*$ and $\Z_o^*$ be these largest even and odd cycles, of orders $3n-1$ and $3n$, respectively.

	Recall that $x$ is the vertex from which $H^\prime$ was extended into $H$ (as illustrated in Figure~\ref{fig:operations}), and let
    \begin{displaymath}
		\J = \{ S\subseteq V(H^\prime) \ \vert\ S \text{ is a dominating set of } H^\prime - \{x\} \text{ but not of } H^\prime \}.
    \end{displaymath}
	We can extend any element of $\J$ in several ways to attain dominating sets of $H$.
	For any $S \in \J$,
    \begin{displaymath}
		S^u, \,
		S^{u,v}, \,
		S^{u,x}, \,
		S^{v,x}, \, \text{ and } \,
		S^{u,v,x}
    \end{displaymath}
	are each dominating sets of $H$.
	Observe that $S \cup \{x\}$ dominates $H^\prime$, and thus $S^{u,x}$, $S^{v,x}$, and $S^{u,v,x}$ are already accounted for in the previously enumerated dominating sets.
	The only new (i.e.\ not already identified) dominating sets are those of the forms $S^u$ and $S^{u,v}$.
	We can use this fact to augment our previously established cycles by adding these new sets.

	Recall that $(G_j)_{1 \leq j \leq n}$ is a cycle of length $n$ in $\ve(H^\prime)$.
	Since $x$ is not an isolated vertex in $H^\prime$,
	there is at least one $G_j$ for which $G_j \neq S^x$ for any $S \in \J$.
	We can assume without loss of generality that $G_1$ is one such $G_j$.
	Now, for each $S \in \J$, recall that $S^x = G_j$ for some $j$ satisfying $2 \leq j \leq n$.
	If $j$ is even, note that the cycle $\Z_o^*$
	will contain the edge \begin{math}(G_j^u, G_j^{u,v}) = (S^{u,x}, S^{u,v,x})\end{math}.
	We can replace this edge with the path \begin{math}(S^{u,x}, S^u, S^{u,v}, S^{u,v,x})\end{math}.
	On the other hand, if $j$ is odd,
	$G_j^{u,v}$ will immediately precede $G_j^u$,
	and so the edge \begin{math}(G_j^{u,v}, G_j^u) = (S^{u,v,x}, S^{u,x})\end{math},
	can be similarly replaced with \begin{math}(S^{u,v,x}, S^{u,v}, S^u, S^{u,x})\end{math}.
	Either way, the length of the cycle is increased by two.

	We can once again apply this operation for each $S \in \J$
	to both $\Z_o^*$ and $\Z_e^*$ to obtain all odd cycles of lengths
	up to $3n + 2\abs{\J}$ and all even cycles of lengths up to $3n-1 + 2\abs{\J}$,
	unless $G_n = S^x$ for some $S \in \J$.
	In this case, note that $\Z_e^*$ does not contain the vertex $G_n^{u,v}$,
	and therefore does not contain the edge $(G_n^{u,v}, G_n^u)$.
	Thus we can only construct the even cycles of lengths up to $3n-3 + 2\abs{\J}$ in this way.
	For the cycle of length $3n-1 + 2\abs{\J}$,
	take instead the cycle of length $3n + 2\abs{\J}$ we just created from $\Z_o^*$,
	and replace the path \begin{math}(G_n^v = S^{v,x}, G_n^{u,v} = S^{u,v,x}, S^{u,v}, S^u, G_n^u = S^{u,x})\end{math}
	with \begin{math}(G_n^v = S^{v,x}, S^{u,v}, G_n^{u,v} = S^{u,v,x}, G_n^u = S^{u,x})\end{math}.
	This is a cycle of length $3n-1+2\abs{J}$, and so we have now found cycles of all lengths up to $3n+2\abs{\J}$.

	Indeed, there are no other dominating sets of $H$;
	only those extended from $\J$ or extended from dominating sets of $H^\prime$.
	Thus $\ve(H)$ is pancyclic.
\end{proof}

It can be noted that among all of the newly established cycles,
only at most three more TS edges were used than were present in the cycle $(G_j)_{1\leq j\leq n}$.
But the number of vertices of $\ve(H)$ is more than triple that of $\ve(H^\prime)$,
and so the proportion of TS edges used stays very small with respect to $n$.

We now prove a similar result for Operation B.


\begin{theorem}
	Let $H$ and $H^\prime$ be graphs such that $H^\prime$ is obtained from $H$ by Operation B.
	If $\ve(H^\prime)$ is pancyclic, then $\ve(H)$ is also pancyclic.
\end{theorem}

\begin{proof}
	We begin by observing that every dominating set of $H$
	will belong to exactly one of the following five sets:
	
	\begin{itemize}
		\item \begin{math}\X^\prime = \{S \in V(\ve(H)) \ \vert\ x \in S, u \not\in S, v \not\in S\}\end{math}
		\item \begin{math}\X = \{S \in V(\ve(H)) \ \vert\ x \in S, u \not\in S, v \in S\} = \{S^v \ \vert\ S \in \X^\prime\}\end{math}
		\item \begin{math}\B^\prime = \{S \in V(\ve(H)) \ \vert\ x \in S, u \in S, v \not\in S\} = \{S^u \ \vert\ S \in \X^\prime\}\end{math}
		\item \begin{math}\B = \{S \in V(\ve(H)) \ \vert\ x \in S, u \in S, v \in S\} = \{S^{u,v} \ \vert\ S \in \X^\prime\}\end{math}
		\item \begin{math}\U = \{S \in V(\ve(H)) \ \vert\ x \not\in S, u \in S, v \in S\}\end{math}.
	\end{itemize}
	Further, each of $\X^\prime$, $\X$, $\B^\prime$, and $\B$ have equal cardinalities,
	and there is a natural bijection between each
	by including or excluding the vertices $u$ and $v$ from each dominating set.
	Throughout this proof, if $X_j$ is some element of $\X$,
	then the symbols $X_j^\prime$, $B_j$, and $B_j^\prime$ will be used to denote
	the corresponding dominating sets in $\X^\prime$, $\B$, and $\B^\prime$,
	and likewise in the other directions.
	That is, \begin{math}X_j^\prime = X_j \setminus \{v\}\end{math},
    \begin{math}B_j = X_j \cup \{u\}\end{math},
    and \begin{math}B_j^\prime = X_j \cup\{u\} \setminus\{v\}\end{math}.

	As in the proof of the previous theorem, since $\ve(H^\prime)$ is pancyclic, there exist cycles in $\ve(H^\prime)$
	of every length from $3$ to $n = \abs{V(\ve(H^\prime))}$.
	We can extend each dominating set of these cycles by $v$ to obtain dominating sets of $\ve(H)$,
	and thus we have cycles in $\ve(H)$ of each length from $3$ to $n$ as well.
	In particular, let $\Z_e$ be a cycle of length $n-1$ in $\ve(H)$
	and $\Z_o$ a cycle of length $n$ in $\ve(H)$,
	each of which correspond to a cycle of the same length in $\ve(H^\prime)$.
	Observe that $\Z_o$ contains exactly the elements of $\U \cup \B \cup \X$,
    and so \begin{math}n = \abs{\U} + \abs{\B} + \abs{\X}\end{math}.

	Suppose \begin{math}\Q = (X_1, X_2, \ldots, X_\ell)\end{math} is a maximal subpath of $\Z_o$
	such that every vertex of $\Q$
	is an element of $\X$.
	For each pair of adjacent vertices $(X_j, X_{j+1})$ in $\Q$ where $j$ is odd,
	we can replace them in $\Z_o$ with the path \begin{math}(X_j, X_j^\prime, X_{j+1}^\prime, X_{j+1})\end{math}
	and extend the length of our cycle by two.
	We can further replace
	$(X_j^\prime, X_{j+1}^\prime)$ with \begin{math}(X_j^\prime, B_j^\prime, B_{j+1}^\prime, X_{j+1}^\prime)\end{math} in $\Z_o$
	to extend the length by another two vertices.
	Together, these two replacements are called a \emph{detour}.
	If $\ell$ is odd, we may also replace $X_\ell$ with \begin{math}(X_\ell, X_\ell^\prime, B_\ell^\prime)\end{math} in $\Z_o$
	to extend by another two vertices, as any element of $\U$ or $\B$ that is adjacent
	to $X_\ell$ will also be adjacent to \begin{math}B_\ell^\prime = X_\ell \cup \{u\} \setminus \{v\}\end{math}.
	We call such a replacement at an endpoint of a path a \emph{special detour}.

	As each replacement increments the length of the cycle by two,
	performing these replacements iteratively for all pairs of vertices in $\Q$
	(and also the last vertex of $\Q$, if $\ell$ is odd),
	for all maximal subpaths $\Q$ with the desired property
	will yield all cycles of odd length from $n$ to $n+2\abs{\X}$.
	But this is exactly the number of the dominating sets of $H$
	since \begin{math}\abs{V(\ve(H))} = \abs{\X} + \abs{\X^\prime} + \abs{\B} + \abs{\B^\prime} + \abs{\U}\end{math} and \begin{math}\abs{\X} = \abs{\X^\prime} = \abs{\B} = \abs{\B^\prime}\end{math}.
	So, this process has given us a cycle of every odd length greater than or equal to $3$.

	Replacements in even cycles follow a similar process---repeating these replacements for the vertices
	of the maximal subpaths of $\Z_e$ contained entirely in $\X$
	will also increase the length by two each time,
	obtaining all even cycles up to a maximum of $n-1$ plus twice the number of vertices
	$\Z_e$ shares with $\X$.
	If $\Z_e$ contains every element of $\X$,
	then this is simply $n-1 + 2\abs{\X}$,
	which is all possible even cycles, and thus $\ve(H)$ is pancyclic.

	On the other hand, suppose that $\Z_e$ does not contain every element of $\X$.
	Recall that $\Z_e$ contains only elements from $\U \cup \B \cup \X$, and, in particular, contains all but one of these elements, since \begin{math}\abs{\Z_e} = n-1\end{math}.
	So, suppose that $X_1$ is the unique element in $\X$ but not in $\Z_e$.
	In this case, replacements can only extend the length of $\Z_e$ up to \begin{math}n-1+2(\abs{\X}-1) = n+2\abs{\X}-3\end{math}. Thus, we are missing a single cycle of length $n+2\abs{\X}-1$ in $\ve(H)$. To complete the proof, we extend $\Z_e$ in a different way in order to obtain a cycle of length $n+2\abs{\X}-1$ in $\ve(H)$.

	Since $\Z_e$ is only missing a single element of $\U \cup \B \cup \X$---namely, $X_1$---it must be the case that $B_1$ is in $\Z_e$. Furthermore, since $X_1$ is the only neighbour of $B_1$ in $\X$, the elements immediately preceding and succeeding $B_1$ in the cycle $\Z_e$ are elements from $\U$ or $\B$. There are two cases. \\

	\textbf{Case 1}: Suppose that either the vertex immediately preceding $B_1$ in $\Z_e$ or the vertex immediately succeeding $B_1$ in $\Z_e$ is in $\B$.
	Denote this vertex $B_2$
	and assume without loss of generality that $B_1$ precedes $B_2$.
	Replacing the edge $(B_1, B_2)$
	with the path \begin{math}(B_1, B_1^\prime, X_1^\prime, X_2^\prime, B_2^\prime, B_2)\end{math}
	adds in the elements of $\B^\prime$ and $\X^\prime$ corresponding to
	$B_1$ and $B_2$,
	but does not include the corresponding elements from $\X$.
	Since \begin{math}\abs{\Z_e}=n-1\end{math} and $X_1$ is not in $\Z_e$, we know $X_2$ must appear in $\Z_e$.
	Now consider the maximal subpath $\Q^*$ of $\Z_e$ that contains only elements of $\X$,
	and in particular contains $X_2$.
	We cannot do quite the same replacements in $\Q^*$ as we did earlier in the proof,
	since $X_2^\prime$, $B_2^\prime$, and $B_2$ have already appeared in our extended cycle.
	Instead, we partition $\Q^*$ into three parts. Let $\Q_1$ and $\Q_2$ be two subpaths of $\Q^*$ such that \begin{math}\Q^* = (\Q_1, X_2, \Q_2)\end{math}.
	Now perform the replacements from earlier in the proof on $\Q_2$,
	and, if $\Q_1$ has even length, perform them on $\Q_1$ as well.
	If $\Q_1$ has odd length, we can still perform the same replacements,
	but the special detour must happen at the beginning, rather than at the end.
	For all other maximal subpaths $\Q$ contained in $\X$,
	we can do as before.
	This new cycle performs a replacement for every vertex of $\X$
	except $X_1$, $X_2$,
	and additionally includes each of the otherwise missing
	$X_1^\prime$, $B_1^\prime$, $X_2^\prime$, $B_2^\prime$. 
	So, the only missing vertex is $X_1$,
	as desired. \\

 \textbf{Case 2}:
	Otherwise, suppose that both the vertex immediately preceding $B_1$ in $\Z_e$ and the vertex immediately succeeding $B_1$ in $\Z_e$ are elements of $\U$.
	Note that there is only one TAR edge from $B_1$ to anything in $\U$,
	and it is to $B_1 \setminus \{x\}$.
	So, if the vertices on either side of $B_1$ are both in $\U$,
	then one of the adjacencies must be a TS edge.
	Let this be to a vertex $U_2$,
	and assume without loss of generality that $B_1$ precedes $U_2$.
	Since $B_1$ and $U_2$ are adjacent via a TS edge,
	then it must be the case that \begin{math}U_2 = B_1 \cup \{w\} \setminus \{x\}\end{math},
	for some vertex $w \in N(x)$ not equal to $u$ or $v$.
	Since $B_1 \cup \{w\} \setminus \{x\}$ is a dominating set,
	$B_1 \cup \{w\}$ is also a dominating set,
	which we will denote as $B_2$.
	Similarly we can define dominating sets $X_2$, $X_2^\prime$, and $B_2^\prime$
	to correspond with $B_2$,
	and note that these will each be adjacent to their counterparts
	$X_1$, $X_1^\prime$, and $B_1^\prime$ via a TAR edge, respectively.
	We replace the edge
	$(B_1, U_2)$ with the path
	\begin{math}(B_1, X_1, X_1^\prime, B_1^\prime, B_2^\prime, U_2)\end{math}
	to include four extra vertices.
	Then, performing the replacements from earlier in the proof
	for each maximal subpath $\Q$ of $\Z_e$ contained entirely in $\X$
	will yield a cycle of length $n + 2\abs{\X} + 1$,
	where the vertex $B_2^\prime$ occurs twice.
	However, the other occurrence of $B_2^\prime$
	will be as a replacement inside one such maximal subpath $\Q^*$.
	In particular, it will be either inside a detour
	\begin{math}(X_2, X_2^\prime, B_2^\prime, B_j^\prime, X_j^\prime, X_j)\end{math}
	for some $X_j \in \X$ adjacent to $X_2$,
	or inside a special detour \begin{math}(X_2, X_2^\prime, B_2^\prime, S)\end{math},
	where $S$ is a dominating set in $\U \cup \B$ adjacent to $X_2$.
	In the former case, $B_2^\prime$ can be eliminated
	by instead replacing $(X_2, X_j)$ with
	\begin{math}(X_2, X_2^\prime, X_j^\prime, X_j)\end{math},
	and in the latter case, the special detour can be entirely avoided,
	simply keeping the $(X_2, S)$ edge intact.
	In either case, we save two of the extra added vertices,
	bringing this cycle down to a length of $n + 2\abs{\X} - 1$,
	the desired length.

	Thus, $\ve(H)$ is pancyclic.
\end{proof}

While these theorems apply to many different graphs and classes of graphs, they were chosen specifically for their application to trees.
In particular, every nontrivial tree can be reduced to the single edge $P_2$ by performing a sequence of Operations A and B,
as is shown in~\cite{ABCHMSS22}.
We can see that $\ve(P_1) \cong P_1$ is trivially pancyclic,
and $\ve(P_2) \cong K_3$ is also pancyclic,
and thus we obtain the following corollary:

\begin{corollary}
	\label{cor:trees}
	For any finite tree $T$, $\ve(T)$ is pancyclic.
\end{corollary}


\section{Combining graphs}\label{sec:join}

There are many ways in which graphs can be combined. We provide a few of them here:

The \emph{union} of two graphs $G$ and $H$, denoted $G \cup H$, is the graph whose vertex and edge sets are the unions of the vertex and edge sets of the underlying graphs, respectively.
The \emph{Cartesian product} of two graphs $G$ and $H$, denoted $G\ \square\ H$, is the graph whose vertex set is the Cartesian product of the vertex sets of $G$ and $H$, where two vertices $(u,v)$ and $(x,y)$ are adjacent if either \begin{math}u = x\end{math} and \begin{math}v \sim y\end{math}, or if \begin{math}u \sim x\end{math} and \begin{math}v = y\end{math}.
The \emph{join} of two graphs $G$ and $H$, denoted $G \vee H$, is the union of $G$ and $H$, along with every possible edge between a vertex $h \in V(H)$ and a vertex $g \in V(G)$.

We denote by $d_G(x,y)$, the distance from $x$ to $y$ in the graph $G$. According to~\cite{SLTH09},
a bipartite graph $G$ is \emph{$k$-cycle bipanpositionable} if,
for any two distinct vertices $x$ and $y$,
and any integer $\ell$ which has the same parity as $d_G(x,y)$ and satisfies $d_G(x,y) \leq \ell \leq \frac{k}{2}$,
there exists a cycle $C$ of length $k$ in $G$ with $d_C(x,y)=\ell$.
Further, a bipartite graph $G$ is \emph{bipanpositionable bipancyclic} if $G$ is $k$-cycle bipanpositionable
for every even integer $k$ between $4$ and $\abs{V(G)}$.
\cite{SLTH09} prove that for any integer $n \geq 2$, the hypercube $Q_n$ is bipanpositionable bipancyclic, where the hypercube $Q_n$ is the graph on $2^n$ vertices obtained by taking the Cartesian product of $n$ copies of $P_2$.
This important fact will be of use several times throughout this section.

We begin by noting that if a graph has multiple components, a dominating set of the graph must dominate each component individually, which implies the following:

\begin{proposition}\label{prop1}
	Let $H_1$ and $H_2$ be graphs. If $H = H_1 \cup H_2$ is the union of $H_1$ and $H_2$, then \begin{math}\ve(H) \cong \ve(H_1)\ \square\ \ve(H_2)\end{math}.
\end{proposition}

The following result of~\cite{Nishi} on the pancyclicity of Cartesian products will also be useful to note.

\begin{theorem}\label{thm:nishi}[Theorem 1.3 in~\cite{Nishi}] Let $G_1$ be a pancyclic graph with an odd number of vertices and let $G_2$ be a graph with a Hamilton path. Then $G_1\ \square\ G_2$ is pancyclic.
\end{theorem}

Let $H_1$ and $H_2$ be graphs such that $\varepsilon(H_1)$ and $\varepsilon(H_2)$ are both pancyclic.
\cite{BCS09} proved that every graph has an odd number of dominating sets; thus $\varepsilon(H_1)$ and $\varepsilon(H_2)$ both contain an odd number of vertices. Moreover, as $\varepsilon(H_1)$ and $\varepsilon(H_2)$ are both pancyclic, they both contain Hamilton paths. Hence the next corollary follows from Proposition~\ref{prop1} and Theorem~\ref{thm:nishi}.

\begin{corollary}\label{cor:union}
	Let $H$, $H_1$, and $H_2$ be graphs such that $H = H_1 \cup H_2$.
	If $\ve(H_1)$ and $\ve(H_2)$ are both pancyclic, then $\ve(H)$ is also pancyclic.
\end{corollary}

In essence, to determine whether a TARS-graph is pancyclic, it suffices to consider only the components of the seed graph individually, as that is enough information to determine whether the TARS-graph itself is pancyclic.
In fact, this allows us to relax the hypotheses of Theorem~\ref{thm:opA}---we no longer require connectedness of the seed graph:

\begin{corollary}
	Let $H$ and $H^\prime$ be graphs such that $H^\prime$ is obtained from $H$ by Operation A.
	If $\ve(H^\prime)$ is pancyclic, then $\ve(H)$ is also pancyclic.
\end{corollary}

The proof is exactly the same as the proof of Theorem~\ref{thm:opA}, with one exception.
In Theorem~\ref{thm:opA}, we require additionally that $H^\prime$ be a connected graph on at least $2$ vertices so that the operation does not leave the vertex $x$ isolated.
However, if $x$ \emph{were} isolated in $H^\prime$, note that
\begin{align*}
	\ve(H^\prime)
	&\cong \ve\left(\left(H^\prime - \{x\}\right) \cup \{x\}\right) \\
	&\cong \ve(H^\prime - \{x\}) \ \square\ \ve(K_1) \\
	&\cong \ve(H^\prime - \{x\}) \ \square\ K_1 \\
	&\cong \ve(H^\prime - \{x\})
\end{align*}
since \begin{math}\ve(K_1) \cong K_1\end{math}.
Further, \begin{math}\ve(H) \cong \ve( \left(H^\prime - \{x\}\right) \cup P_3) \cong \ve(H^\prime \cup P_3)\end{math}.
It can be verified that $\ve(P_3)$ is pancyclic, and thus so is $\ve(H)$.

\begin{corollary}
	If $F$ is a forest, then $\ve(F)$ is pancyclic.
\end{corollary}

We next consider the join operation.

\begin{lemma}
	\label{lemma:k1join}
	Let $G$ be any graph. If $\ve(G)$ is pancyclic, then $\ve(G \vee K_1)$ is also pancyclic.
\end{lemma}

\begin{proof}
	For a graph $G$, let \begin{math}n = \abs{V(G)}\end{math} and \begin{math}N = \abs{V(\ve(G))}\end{math}.
	If $G$ contains no edges, then \begin{math}G \vee K_1 \cong K_{n,1}\end{math}, and $\ve(K_{n,1})$ is pancyclic by Corollary~\ref{cor:trees}.
	Consequently, suppose $G$ contains at least one edge.
	Let $x$ denote the sole vertex of $K_1$, and define \begin{math}H = G \vee K_1\end{math}.
	Note that the dominating sets of $H$
	are exactly the dominating sets of $G$,
	along with $\{x\} \cup T$ for every subset $T$ of $V(G)$.
	Observe that the latter dominating sets form a hypercube $Q_n$
	under the TAR edges.

	Let $v$ be a fixed, non-isolated vertex in $G$, and let $S = V(G) \setminus \{v\}$.
	Reusing the notation from Section~\ref{sec:trees},
	let $S^v$ denote $S \cup \{v\}$, $S^x$ denote $S \cup \{x\}$, and $S^{v,x}$ denote $S \cup \{v,x\}$.
	Note that $S^{v,x}$ and $S^x$ are elements of $Q_n$,
	and that $S^v$, $S^{v,x}$, and $S^x$ form a triangle in $\ve(H)$.
	Since, by~\cite{SLTH09}, $Q_n$ is bipanpositionable bipancyclic 
	for every even length $\ell$ from $4$ to $2^n$,
	there is a cycle of length $\ell$ in $Q_n$
	which places $S^{v,x}$ and $S^x$ contiguously.
	So we have cycles in $\ve(H)$ of every even length between $4$ and $2^n$, inclusive.
	In each such cycle,
	we can replace the edge $(S^{v,x}, S^x)$ with the path \begin{math}(S^{v,x}, S^v, S^x)\end{math},
	and thus obtain a cycle of every odd length between $5$ and $2^n+1$, inclusive.
	Observe that replacing the edge $(S^{v,x}, S^x)$ in the cycle of length $2^n$
	with \begin{math}(S^{v,x}, S^v, S, S^x)\end{math} instead also yields a cycle of length $2^n+2$.

	As we have already pointed out the triangle in $\ve(H)$,
	we can see that there exist cycles of every length from $3$ to $2^n+2$.
	Then, for any cycle of length $\ell \geq 3$ in $\ve(G)$,
	consider some pair of adjacent vertices $A, B$ in the cycle.
	Observe that $A^x = A \cup \{x\}$ and $B^x = B \cup \{x\}$
	are adjacent dominating sets in $Q_n$,
	and are also adjacent to $A$ and $B$, respectively.
	Since, by \cite{SLTH09}, $Q_n$ is bipanpositionable bipancyclic,
	there exists a cycle through $Q_n$ of length $2^n$ which places $A^x$ and $B^x$ contiguously.
	In particular, there exists a path of length $2^n$ through $Q_n$ from $A^x$ to $B^x$.
	Inserting this path in between $A$ and $B$
	in the cycle of length $\ell$ adds $2^n$ vertices to the cycle,
	thus obtaining a new cycle of length $\ell + 2^n$ through $\ve(H)$
	for each integer $\ell$ such that $3 \leq \ell \leq N$.
	We have exhibited a cycle in $\ve(H)$ of every length
	up to $N+2^n$,
	and therefore $\ve(H)$ is pancyclic.
\end{proof}

Since $K_{n+1} = K_n \vee K_1$,
we obtain the following as an immediate corollary by induction:

\begin{corollary}
	\label{cor:complete}
	For any positive integer $n$, $\ve(K_n)$ is pancyclic.
\end{corollary}

Similarly, a threshold graph can be constructed from a single vertex by repeatedly adding an isolated vertex (i.e.\ taking the union with $K_1$) or a universal vertex (i.e.\ taking the join with $K_1$). The next result then follows from Corollary~\ref{cor:union} and Lemma~\ref{lemma:k1join} in a similar way:

\begin{corollary}
	If $G$ is a threshold graph, $\ve(G)$ is pancyclic.
\end{corollary}

We next generalize Lemma~\ref{lemma:k1join}.

\begin{theorem}
	\label{thm:joins}
	Let $G$ and $H$ be graphs. If $\ve(G)$ and $\ve(H)$ are pancyclic, then $\ve(G \vee H)$ is also pancyclic.
\end{theorem}

We refer the reader to Figure~\ref{fig:joins} for a schematic diagram which illustrates the structure of the various cycles used in the following proof.

\begin{proof}
	Let $\abs{V(G)} = m$ and $\abs{V(H)} = n$ and assume without loss of generality that $m \geq n$. Given Lemma~\ref{lemma:k1join}, we assume $H \neq K_1$,
	and thus $n \geq 2$.
	In each of the cases where $m=n=2$,
	$G \vee H$ is isomorphic to either $K_4$, $P_3 \vee K_1$, or $C_4$,
	and in these cases $\ve(G \vee H)$ is pancyclic by Corollary~\ref{cor:complete},
	Lemma~\ref{lemma:k1join}, or by inspection, respectively.
	Therefore we can further assume $m \geq 3$.
	We can also assume that $G$ is not the complete graph $K_m$,
	as \begin{math}K_m = K_1 \vee K_1 \vee \cdots \vee K_1\end{math},
	and so \begin{math}\ve(G \vee H) \cong \ve(H \vee K_1 \vee \cdots \vee K_1)\end{math}
	will be pancyclic by induction.
	Thus there exist at least two vertices $g$ and $g^\prime$ in $G$ of degree at most $m-2$.
	Similarly, assume $h$ is a vertex of degree at most $n-2$ in $H$.

	Let \begin{math}\XX = (x_0, x_1, \ldots, x_{2^m-1})\end{math} be a binary-reflected Gray code
	through all subsets of the vertices of $G$
	such that \begin{math}x_1 = \{g\}\end{math}. Note that \begin{math}x_1=\{g\}\end{math} corresponds to the binary string with $0$ in each entry except for the entry corresponding to vertex $g$.
	Observe that by construction,
	$x_{2^m-2}$ will also contain $g$, as well as a single other vertex, which we pick to be $g^\prime$.
	Similarly, let \begin{math}\YY = (y_0, y_1, \ldots, y_{2^n-1})\end{math} be a binary-reflected Gray code
	through the subsets of the vertices of $H$,
	such that \begin{math}y_1 = \{h\}\end{math}.

	Observe that every subset of $V(G \vee H)$
	is the union of a subset of $V(G)$ and a subset of $V(H)$. We can thus express each subset of $V(G \vee H)$ as an ordered pair $(x_i, y_j)$,
	for some nonnegative integers $i$ and $j$ at most $2^m-1$ and $2^n-1$, respectively.

	Note that \begin{math}(x_1,y_1) = \{g,h\}\end{math} is a dominating set of $G \vee H$,
	and so $\{g,h\} \cup S$ is also a dominating set of $G \vee H$
	for any subset $S$ of $V(G \vee H) \setminus \{g,h\}$.
	In particular, the $2^{m+n-2}$ elements of
	\begin{math}\{\{g,h\} \cup S \ \vert\ S \subseteq V(G \vee H) \setminus \{g,h\}\}\end{math}
	form a hypercube $Q_{m+n-2}$ in $\ve(G \vee H)$ under only the TAR edges.
	Observe that the sets $V(G \vee H)$, $V(G \vee H) \setminus \{h\}$, and $V(G \vee H) \setminus \{g^\prime\}$
	form a triangle in $\ve(G \vee H)$,
	and $V(G \vee H)$ and $V(G \vee H) \setminus \{g^\prime\}$ are both elements of $Q_{m+n-2}$.
	Since, by~\cite{SLTH09}, hypercubes are bipanpositionable bipancyclic,
	there is a cycle of every even length in $Q_{m+n-2}$ from $4$ to $2^{m+n-2}$
	which places $V(G \vee H)$ and $V(G \vee H) \setminus \{g^\prime\}$ contiguously.
	As $V(G \vee H)$ and $V(G \vee H) \backslash \{g^\prime\}$ are both adjacent to $V(G \vee H) \setminus \{h\}$,
	we can insert $V(G \vee H) \backslash \{h\}$ between them in every even cycle to extend the length by one.
	Thus, $\ve(G \vee H)$ exhibits cycles of every length up to $2^{m+n-2}+1$.

	Since \begin{math}(x_1,y_1) \sim (x_{2^m-2},y_1)\end{math}, the path
	\begin{equation}
		\label{eq:joins_cycle}
		\left( (x_1,y_1), (x_2,y_1), (x_3,y_1), \ldots, (x_{2^m-2},y_1) \right)
	\end{equation}
	along the Gray code $\XX$ is a cycle of length $2^m-2$ in $\ve(G \vee H)$.
	Suppose that $t$ is the nonnegative integer such that \begin{math}x_t = V(G)\end{math}.
	Inserting the vertex $(x_t,y_0)$ either after $(x_t,y_1)$ if $t$ is even, or before $(x_t,y_1)$ if $t$ is odd
	will create a cycle of length $2^m-1$ in $\ve(G \vee H)$.
	The parity of $t$ is important so that we do not remove the edges necessary for the coming replacements.
	Then, for every edge $( (x_{2i+1}, y_j), (x_{2i+2}, y_j) )$
	with $0 \leq i \leq 2^{m-1} - 2$
	and $1 \leq j \leq 2^n - 2$ in either of these two cycles,
	replacing that edge with the path
	\begin{math}( (x_{2i+1}, y_j), (x_{2i+1}, y_{j+1}), (x_{2i+2}, y_{j+1}), (x_{2i+2}, y_j) )\end{math}
	increases the length of the cycle by two.
	These extensions are depicted in \textcolor{cyan}{cyan} in Figure~\ref{fig:joins}.
	Repeating this replacement for every possible edge in the original two cycles
	yields a cycle of every length up to \begin{math}(2^m-2)(2^n-1)+1 = 2^{m+n} - 2^m - 2^{n+1} + 3\end{math}.

	To construct the remaining cycles, we start with two new cycles:
	Let $\Z_e$ be the cycle of length $2^m-2$ defined in~(\ref{eq:joins_cycle}),
	but with the following replacements:
	\begin{enumerate}[label={(\alph*)}]
		\item replace the edge $((x_{2^m-3},y_1), (x_{2^m-2},y_1))$
			with the path from $(x_{2^m-3},y_1)$ to $(x_{2^m-3},y_{2^n-1})$ along the Gray code $\YY$,
			followed by $(x_{2^m-2},y_{2^n-1})$,
			and then the path from $(x_{2^m-1},y_{2^n-1})$ to $(x_{2^m-1},y_1)$ along the Gray code $\YY$ in reverse,
			followed by $(x_{2^m-2},y_1)$; then
		\item replace every edge $((x_{2^m-3}, y_{2j}), (x_{2^m-3}, y_{2j+1}))$
			for $j$ satisfying $1 \leq j \leq 2^{n-1}-2$,
			with the path \begin{math}((x_{2^m-3}, y_{2j}), (x_{2^m-2}, y_{2j}), (x_{2^m-2}, y_{2j+1}), (x_{2^m-3}, y_{2j+1}))\end{math}.
	\end{enumerate}

	Let $\Z_o$ be the same as $\Z_e$,
	except replace the edge $((x_{2^m-3},y_{2^n-1}), (x_{2^m-2},y_{2^n-1}))$
	with the path \begin{math}((x_{2^m-3},y_{2^n-1}), (x_{2^m-2},y_{2^n-2}), (x_{2^m-2},y_{2^n-1}))\end{math}.
	The edges in common between $\Z_e$ and $\Z_o$ are depicted in \textcolor{black}{black} in Figure~\ref{fig:joins},
	with the differences depicted in \textcolor{orange}{orange} and \textcolor{teal}{teal}, respectively.

	For each of these two cycles,
	sequentially repeating the same edge replacement procedure as mentioned before
	(i.e., replacing each edge $( (x_{2i+1}, y_j), (x_{2i+2}, y_j) )$
	with the path
    \begin{displaymath}( (x_{2i+1}, y_j), (x_{2i+1}, y_{j+1}), (x_{2i+2}, y_{j+1}), (x_{2i+2}, y_j) )\end{displaymath}
	for each integer $i$ and $j$ satisfying
	$0 \leq i \leq 2^{m-1} - 3$ and $1 \leq j \leq 2^n - 2$)
	will increase the lengths of these cycles by two each time,
	yielding cycles of every length up to \begin{math}(2^m-1)(2^n-1) = 2^{m+n} - 2^m - 2^n\end{math}.
	This is again depicted in \textcolor{cyan}{cyan} in Figure~\ref{fig:joins}.
	In fact, $\Z_o$ with these augmentations
	will contain every dominating set $(x_i,y_j)$ for positive $i,j$.
	That is, this cycle loops through every dominating set which has nonempty intersection with both $G$ and $H$.

	Observe that every dominating set of $G \vee H$ which does not intersect $H$
	can be written as $(x_i,y_0)$ for some positive integer $i \leq 2^m-1$.
	Further note that $1 < i < 2^m-1$, since neither $g$ nor $g^\prime$ are universal vertices.
	Then each dominating set $(x_i,y_0)$ will be adjacent to $(x_i,y_1)$,
	and exactly one of the edges \begin{math}\{((x_{i-1},y_1),(x_i,y_1)), ((x_i,y_1),(x_{i+1},y_1))\}\end{math} will be present within the augmented $\Z_o$.
	Denote the vertices incident with this edge as $(x_i,y_1)$ and $(x_{i \pm 1},y_1)$,
	accordingly.
	If $(x_{i \pm 1}, y_0)$ is also a dominating set,
	then this edge can be replaced with the path
	\begin{math}((x_i,y_1), (x_i,y_0), (x_{i \pm 1},y_0), (x_{i \pm 1},y_1))\end{math} to include both dominating sets.
	On the other hand, if $(x_{i \pm 1}, y_0)$ is not a dominating set,
	then it must be the case that \begin{math}x_{i \pm 1} = x_i \setminus \{v\}\end{math} for some vertex $v \in G$,
	since the binary-reflected Gray code is formed via only TAR edges,
	and supersets of dominating sets are still dominating.
	So then \begin{math}(x_i,y_0) \sim (x_{i \pm 1},y_1)\end{math} via a TS edge,
	since this exchanges $h \in H$ for $v \in G$.
	Thus, we can replace the edge with the path \begin{math}((x_i,y_1), (x_i,y_0), (x_{i \pm 1},y_1))\end{math}.
	These replacements are depicted in \textcolor{magenta}{magenta} in Figure~\ref{fig:joins}.

	The dominating sets that do not intersect $G$ can also be included sequentially
	in a similar manner, with the exception that $(x_1,y_j)$ will be adjacent to both
	$(x_1,y_{j-1})$ and $(x_1,y_{j+1})$ for integers $j$ between $2$ and $2^n-2$, inclusive.
	If $j$ is even, use the edge $((x_1,y_j),(x_1,y_{j+1}))$,
	and if $j$ is odd, use the edge $((x_1,y_j),(x_1,y_{j-1}))$.
	These are also depicted in \textcolor{magenta}{magenta} in Figure~\ref{fig:joins}.

	So, by including each extra dominating set sequentially,
	we can obtain a cycle of any length through $\ve(G \vee H)$,
	and thus $\ve(G \vee H)$ is pancyclic.
\end{proof}

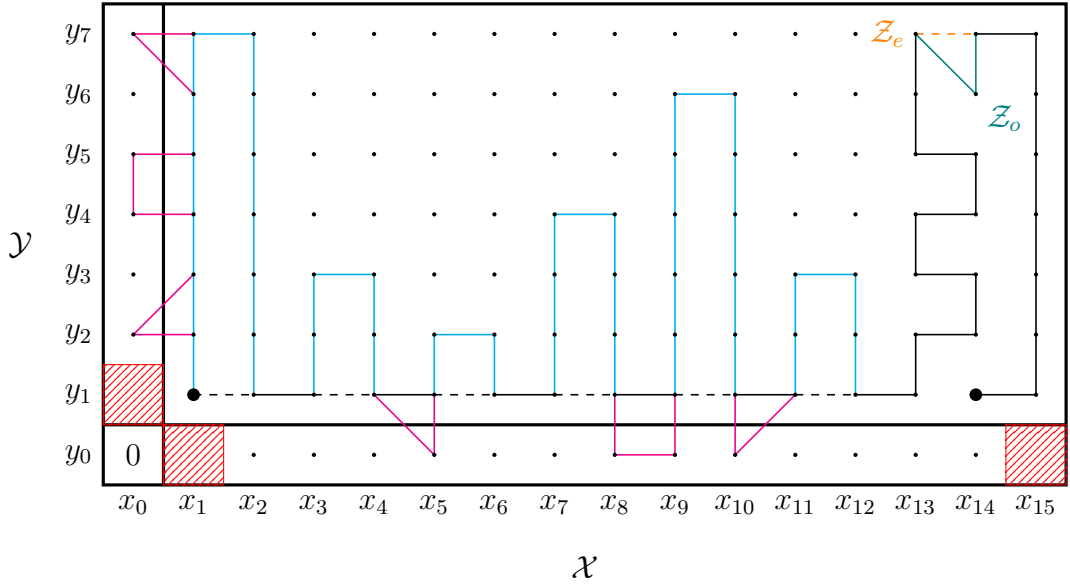
\begin{figure}[hbtp]
	\centering
	\begin{tikzpicture}[scale=0.8]
		\draw[very thick] (-0.5,-0.5) rectangle ++(16,8);
		\draw[very thick] (0.5,-0.5) -- ++(0,8);
		\draw[very thick] (-0.5,0.5) -- ++(16,0);

		\draw[red, pattern=north east lines, pattern color=red] (0.5,-0.5) rectangle ++(1,1);
		\draw[red, pattern=north east lines, pattern color=red] (-0.5,0.5) rectangle ++(1,1);
		\draw[red, pattern=north east lines, pattern color=red] (14.5,-0.5) rectangle ++(1,1);

		\draw[dashed, semithick] (1,1) -- ++(1,0); 
		\draw[dashed, semithick] (3,1) -- ++(1,0);
		\draw[dashed, semithick] (5,1) -- ++(1,0);
		\draw[dashed, semithick] (7,1) -- ++(1,0);
		\draw[dashed, semithick] (9,1) -- ++(1,0);
		\draw[dashed, semithick] (11,1) -- ++(1,0);

		\draw[semithick] (2,1) -- ++(1,0); 
		\draw[semithick] (4,1) -- ++(1,0);
		\draw[semithick] (6,1) -- ++(1,0);
		\draw[semithick] (8,1) -- ++(1,0);
		\draw[semithick] (10,1) -- ++(1,0);
		\draw[semithick] (12,1) -- ++(1,0);

		\draw[semithick] (13,1) -- ++(0,1) -- ++(1,0) -- ++(0,1) -- ++(-1,0) -- ++(0,1) -- ++(1,0) -- ++(0,1) -- ++(-1,0) -- ++(0,2); 
		\draw[semithick] (14,7) -- ++(1,0) -- ++(0,-6) -- ++(-1,0); 

		\draw[orange, dashed, semithick] (13,7) -- ++(1,0); 
		\draw[teal, semithick] (13,7) -- ++(1,-1) -- ++(0,1); 

		\draw[cyan, semithick] (1,1) -- ++(0,6) -- ++(1,0) -- ++(0,-6); 
		\draw[cyan, semithick] (3,1) -- ++(0,2) -- ++(1,0) -- ++(0,-2);
		\draw[cyan, semithick] (5,1) -- ++(0,1) -- ++(1,0) -- ++(0,-1);
		\draw[cyan, semithick] (7,1) -- ++(0,3) -- ++(1,0) -- ++(0,-3);
		\draw[cyan, semithick] (9,1) -- ++(0,5) -- ++(1,0) -- ++(0,-5);
		\draw[cyan, semithick] (11,1) -- ++(0,2) -- ++(1,0) -- ++(0,-2);

		\draw[magenta, semithick] (4,1) -- ++(1,-1) -- ++(0,1); 
		\draw[magenta, semithick] (8,1) -- ++(0,-1) -- ++(1,0) -- ++(0,1);
		\draw[magenta, semithick] (10,1) -- ++(0,-1) -- ++(1,1);

		\draw[magenta, semithick] (1,2) -- ++(-1,0) -- ++(1,1); 
		\draw[magenta, semithick] (1,4) -- ++(-1,0) -- ++(0,1) -- ++(1,0);
		\draw[magenta, semithick] (1,6) -- ++(-1,1) -- ++(1,0);

		\node at (0,0) {$0$};
		\node[below right,teal] at (14,6) {$\Z_o$};
		\node[left,orange] at (13,7) {$\Z_e$};

		\path (-0.5,-1.5) -- ++(16,0) node[midway,below] {$\XX$};
		\foreach \x in {0,...,15}
		\path (\x-0.5,-0.5) -- ++(1,0) node[midway,below] {$x_{\x}$};

		\path (-1.5,-0.5) -- ++(0,8) node[midway,left] {$\YY$};
		\foreach \y in {0,...,7}
		\path (-0.5,\y-0.5) -- ++(0,1) node[midway,left] {$y_{\y}$};

		\foreach \x in {1,...,15}
		\foreach \y in {1,...,7}
		\fill (\x,\y) circle (1pt);

		\foreach \x in {2,...,14}
		\fill (\x,0) circle (1pt);

		\foreach \y in {2,...,7}
		\fill (0,\y) circle (1pt);

		\fill (1,1) circle (3pt);
		\fill (14,1) circle (3pt);
	\end{tikzpicture}
	\caption{A schematic diagram of the cycles used in the proof of Theorem~\ref{thm:joins}.}
	\label{fig:joins}
\end{figure}

Note that among all of of the cycles just constructed,
at most $2^{m-1}+2^{n-1}+1$ TS edges were used at a time.
But the number of TAR edges is on the order of $(2^m-1)(2^n-1)$,
and hence the proportion of TS edges is very small in comparison to the total number of edges in these cycles.

This theorem yields a large number of graphs which admit pancyclic TARS-graphs,
including all complete multipartite graphs and complete split graphs.
In particular, this includes the TARS-graph of $K_{2,2} \cong C_4$ is pancyclic.
This is in stark contrast to the results of~\cite{ABCHMSS22},
who showed that dominating graphs of cycles of length $4n$ never admit Hamilton paths,
let alone Hamilton cycles or pancyclicity.
We have also computationally verified that $\ve(C_8)$ has a Hamilton cycle,
and that $\ve(G)$ is pancyclic for all graphs $G$ of order at most $5$.


\section{Open Questions}\label{sec:questions}

We have shown that TARS-graphs are pancyclic for many different types of seed graph, but we have not shown any results in the negative. Indeed, we have not found any graphs whose TARS-graphs are not pancyclic. This leaves an important question unanswered: Are all TARS-graphs pancylic? Or, if not pancyclic, hamiltonian?
Clearly the inclusion of the TS edges to the dominating graph is a powerful addition, since it is enough to induce pancyclicity in graphs which were otherwise bipartite and non-hamiltonian;~\cite{ABCHMSS21}. However, is there a smaller addition which could have induced such a property? In the proofs in Sections~\ref{sec:trees} and \ref{sec:join}, we only ever make use of a very small number of TS edges at a time. In some sense, the dominating graph $\D(G)$ was already very ``close'' to being hamiltonian. It is interesting to consider how many TS edges are necessary to include in order to obtain pancyclicity or hamiltonicity in the reconfiguration graph. Or, on the other hand, how many could be included without inducing pancyclicity/hamiltonicity?
In~\cite{ABCHMSS21,ABCHMSS22}, the authors present Hamilton paths in $\D(P_n)$ whose endpoints are never TS-adjacent.
However, we have computationally verified that for every graph $G$ of order at most $8$,
there does exist a Hamilton cycle of $\ve(G)$ which uses only a single TS edge
whenever $\D(G)$ has a Hamilton path.
Is this true in general? Do the hamiltonicity and pancyclicity of TARS-graphs really only rely on a single extra TS edge?

\acknowledgements
The authors thank Dr.\ L.\ Teshima for her valuable comments on L. Pipes' honours thesis, from which this work was inspired.  The authors also thank Dr.\ A.\ Burgess for a helpful question which led to the investigation of pancylicity in TARS-graphs rather than only hamiltonicity.  Additionally the authors thank the anonymous referee for their comments and suggestions.

\nocite{*}
\bibliographystyle{abbrvnat}
\bibliography{15637-dmtcs}
\label{sec:biblio}

\end{document}